\documentclass[12pt]{article}
\usepackage{amsfonts}

\usepackage[latin1]{inputenc}
\usepackage{amsmath,amssymb}
\usepackage{latexsym}
\usepackage[active]{srcltx}
\usepackage[
bookmarks=true,         
bookmarksnumbered=true, 
colorlinks=true, pdfstartview=FitV, linkcolor=blue, citecolor=blue,
urlcolor=blue]{hyperref}

\newtheorem{theorem}{Theorem}[section]

\newtheorem{lemma}[theorem]{Lemma}
\newtheorem{proposition}[theorem]{Proposition}

\numberwithin{equation}{section}
\def\R{{\mathbb R}}
\def\E{{{\mathbb E}\,}}

\def\N{{\mathbb N}}

\def\Var{{\mathop {{\rm Var\, }}}}

\def\square{{\vcenter{\vbox{\hrule height.3pt
        \hbox{\vrule width.3pt height5pt \kern5pt
           \vrule width.3pt}
        \hrule height.3pt}}}}

\def\tlint{{- \kern-0.85em \int \kern-0.2em}}
\def\dlint{{- \kern-1.05em \int \kern-0.4em}}

\def\cF{{\cal F}}
\def\Om{{\Omega}}

\def\cF{{\cal F}}

\def\de{{\delta}}

\def\Om{{\Omega}}

\def\de{{\delta}}

\def \eref#1{\hbox{(\ref{#1})}}

\def \eref#1{\hbox{(\ref{#1})}}

\newenvironment{proof}[1][Proof]{\noindent\textit{#1.} }{\hfill \rule{0.5em}{0.5em}}

\begin{document}

\title{Central limit theorem for an additive functional of the fractional Brownian motion II}
\date{\empty }
\author{  David Nualart\thanks{ D. Nualart is supported by the NSF grant
DMS-1208625.} \ \
and \ Fangjun Xu\thanks{F. Xu is supported in part by the Robert Adams Fund.}\\
Department of Mathematics \\
University of Kansas \\
Lawrence, Kansas, 66045 USA}

\maketitle

\begin{abstract}
\noindent We prove a central limit theorem for an additive
functional of the $d$-dimensional fractional Brownian motion
with Hurst index $H\in(\frac{1}{2+d},\frac{1}{d})$, using the method of moments,
extending the result by Papanicolaou, Stroock and Varadhan in the
case of the standard Brownian motion.


\end{abstract}

\section{Introduction}
Let  $\left\{ B(t)=(B^1(t),\dots,
B^d(t)), t\geq 0\right\}$  be a $d$-dimensional fractional Brownian
motion (fBm) with Hurst index $H\in (0,1)$. The local time of $B$, defined as  $L_t(x)=\int_0^t \de(B(s)-x) ds$, for $t\ge 0$ and $x\in \R^d$, where $\de$ is the Dirac delta function,  exists and   is jointly continuous in $t$ and $x$ if $Hd<1$  (see \cite{GH}).  For any
integrable function $f:\R^d \rightarrow \R$, using the scaling property of the fBm and the continuity of the local time,  one can easily show  the
following convergence in law in the space $C([0,\infty))$,
as $n$ tends to infinity
\begin{equation}\label{e.1.1}
\Big(  n^{Hd-1} \int_0^{nt} f(B(s))\, ds\,, t\ge 0\Big) \
\overset{\mathcal{L}}{\rightarrow } \ \Big( L_t(0) \int_{\R^d}
f(x)\, dx\,, t\ge 0\Big)\,.
\end{equation}
 
If we assume that $\int_{\R^d}f(x)\, dx=0$, a central limit theorem holds with a random variance. In order to formulate this theorem, we need to introduce some notation.   Fix a number $\beta>0$ and denote
\[
H_{0}^{\beta }=\Big\{f \in L^{1}(\R^d):
\int_{\R^d}|f(x)| |x|^{\beta } dx<\infty \quad {\rm and}\quad
\int_{\R^d }f(x)\, dx=0\Big\}\,. \]
For any  $f\in H_{0}^{\beta }$, and assuming $\beta \in (0,2)$,    the quantity  (see Lemma 4.1 in \cite{HNX})
\begin{equation}  \label{beta}
          \Vert f\Vert _{\beta }^{2}:=-\int_{\R^{2d}}f(x)f(y)|x-y|^{\beta }\, dx\, dy = c_{\beta,d} \int_{\R^{d}} |\widehat{f}(x)|^2 |x|^{-\beta-d} dx
\end{equation}
is finite and nonnegative, where $\widehat{f}$ denotes the Fourier transform of $f$.

 Then, the following central limit theorem holds.
\begin{theorem}
\label{th1} Suppose $\frac{1}{d+2}<H<\frac{1}{d}$  and  $f\in H^{\frac{1}{H}-d}_0$.
Then
\[
  \Big(  n^{\frac{Hd-1}{2}}\int_{0}^{nt} f(B(s))\, ds\,,
  \ t\ge 0\Big) \ \ \overset{\mathcal{L}}{\longrightarrow } \ \
  \Big(  \sqrt{C_{H,d}} \, \| f\|_{\frac{1}{H}-d}\, W (L_{t}(0))\,, t\ge
  0\Big)
\]
in the space $C([0,\infty ))$, as $n$ tends to infinity,   where $W $ is a real-valued standard Brownian
motion independent of $B$ and
\[
    C_{H,d} = \frac 2 {(2\pi)^\frac{d}{2}}   \int_0^\infty w^{-Hd} \big(1-\exp(-\frac{1}{2w^{2H}})\big)\,
    dw= \frac{2^{1-\frac{1}{2H}}}{(1-Hd)\pi^{\frac{d}{2}}}\Gamma\Big(\frac{Hd+2H-1}{2H}\Big).
\]
\end{theorem}

This theorem has been proved by Hu, Nualart and Xu in the reference \cite{HNX}, in the case where the Hurst parameter $H$ satisfies  $\frac{1}{d+1}<H<\frac{1}{d}$, and it has been conjectured in that paper that the result can be extended to the case
$\frac{1}{d+2}<H \le \frac{1}{d+1}$. The purpose of the present paper is to prove this conjecture. With this aim we will develop a new approach to prove Theorem \ref{th1} based on Fourier analysis. 

Note that the lower bound $\frac 1{d+2}$ is optimal because for $H\le \frac 1{d+2}$ the constant $C_{H,d}$ is infinite.
When  $d=1$ and $H=\frac 12$,  the above theorem was
obtained by Papanicolaou, Stroock and Varadhan in
\cite{PSV} with $C_{\frac 12, 1}=2$.

As in the reference \cite{HNX}, the  proof of Theorem \ref{th1} is based on  the method of moments. In order to handle    the integrals on $[0,t]^{2m}$, with respect to the measure $ds_1  \cdots ds_{2m}$, we   make  the change of variables 
$u_{2k-1}=n(s_{2k}-s_{2k-1})$ and  $u_{2k}=s_{2k}$, $1\le k\le m$. Then, the increments of $B$ in small intervals will be responsible for the independent noise appearing in the limit.  The main novelty of our approach, in comparison with \cite{HNX},  is a new methodology based on Fourier  analysis and an iterative procedure in order to get the right estimates to derive  the tightness of the laws and to show the convergence to zero in the truncation argument.

After some preliminaries in Section 2,  in  Section 3  we prove some technical estimates based on Fourier analysis which play a fundamental role in our approach. Finally, Section 4 is devoted to the 
proof of Theorem  \ref{th1}. Throughout this paper, if not mentioned otherwise, the letter $c$, 
with or without a subscript, denotes a generic positive finite
constant whose exact value is independent of $n$ and may change from
line to line.

\section{Preliminaries}
Let $\left\{ B(t)= (B^1(t),\dots,
B^d(t)), t\geq 0\right\} $ be a $d$-dimensional fractional Brownian
motion with Hurst index $H\in (0,1)$, defined on some probability space $(\Om, \cF, P)$. That is, the components
of $B$ are independent centered Gaussian processes with covariance 
\[
\E\big(B^i(t)B^i(s)\big)=\frac{1}{2}\big(t^{2H}+s^{2H}-|t-s|^{2H}\big).
\]  

The next lemma (see Lemma 2.1 in \cite{HNX}) gives a formula for the moments of the  increments of the process  
 $\left\{W (L_{t}(0)), t\ge 0\right\}$ on disjoint intervals, where $W$ is a real-valued standard Brownian motion
independent of $B$.
\begin{lemma} \label{lema2} Fix a finite number of disjoint intervals 
$(a_i, b_i]$ in $[0,\infty)$, where $i = 1, \dots, N$  and $b_i \leq a_{i+1}$. 
Consider a multi-index $\mathbf{m} = (m_1, \dots, m_N)$, where $m_i \geq 1$ and 
$1 \leq  i \leq N$. Then
\begin{eqnarray}
&&\E \Big(\prod_{i=1}^N \big[W (L_{b_i} (0) ) - W (L_{a_i} (0))\big]^{m_i} \Big) \label{2.1} \\
&=&\begin{cases} 
\Big(\prod\limits_{i=1}^N \frac{m_i!}{ 2^{\frac{m_i}{2}}   (2\pi )^{\frac{m_id}{4} }   (m_i/2)!} \Big)
\displaystyle \int_{\prod\limits_{i=1}^N [a_i ,b_i]^{\frac{m_i}{ 2}}} \det(A(w))^{-\frac{1}{2}}\, dw
&\hbox{if all $m_i$ are even}\\
0 & \hbox{otherwise}, \\
\end{cases}   \notag
\end{eqnarray}
where $A(w)$ is the covariance matrix of the Gaussian random vector 
\[
\Big(B(w^i_k): 1\leq i\leq N\; \text{and}\; 1\leq k\leq \frac{m_i}{2}\Big).
\]
\end{lemma}
As a consequence, the law of the random vector  $\big(W (L_{b_i} (0) ) - W (L_{a_i} (0)): 1\le i\le N\big)$ is determined by
the moments computed in the above lemma.

We shall use the following local nondeterminism property of the fractional
Brownian motion (see \cite{Be}):  For any $n\ge 2$ there exists a positive constant $k_H$ depending on $n$, such that for any $0=s_0<s_1\leq \dots \leq
s_n<\infty$ and $ u_1,\dots, u_n\in\R^d$,  
\begin{equation}
\Var\Big(\sum^n_{i=1} u_i\cdot \big(B(s_i)-B(s_{i-1})\big)\Big)\geq
k_H\sum^n_{i=1}|u_i|^2(s_i-s_{i-1})^{2H}\,.\label{lndp}
\end{equation}

\section{Technical estimates}

We are interested in the sequence of stochastic processes defined by
\[
F_n(t)= n^{\frac{Hd-1}{2}}\int_{0}^{nt} f(B(s))\, ds.
\]
For $0\leq a<b<\infty$ and $m\in\N$, let  $I^n_m=\E\left[ (F_n(b) -F_n(a) ^m\right]$. It is easy to see that
\begin{align*}
I_m^n& = m!\, n^{m\frac{  Hd-1 }{2}} \int_{D_m} \E \Big ( \prod_{i=1}^m f(B(s_i)) \Big)\, ds \\
&= c_{m,d}\,  n^{m\frac{ Hd-1 }{2}}  \int_{ \R ^{md}}  \int_{D_m}  \Big(\prod_{i=1}^m\widehat{f}(y_i) \Big)\,  
 \exp\Big(-\frac{1}{2}\Var\big(\sum\limits^m_{i=1} y_i\cdot B(s_i)\big)\Big)\, ds\, dy,
 \end{align*}
where $c_{m,d}=\frac{m!}{(2\pi)^{md}}$ and $D_m=\big\{(s_1, \dots, s_m): na<s_1<\cdots<s_m<nb\big\}$. Making the change of variables
$x_i =\sum\limits_{j=i}^m y_j$ (with the convention that $x_{m+1} =0$) we can write
\begin{align*}
I^n_m
&=c_{m,d}\, n^{m\frac{Hd-1}{2}}\int_{\R^{md}}\int_{D_m} \Big(\prod^m_{i=1} \widehat{f}(x_i-x_{i+1})\Big)\\
&\qquad\qquad\times \exp\Big(-\frac{1}{2}\Var\big(\sum\limits^m_{i=1} x_i\cdot (B(s_i)-B(s_{i-1}))\big)\Big)\, ds\, dx.
\end{align*}
The main idea in order to estimate these terms is to replace each product  $\widehat{f}(x_{2i-1}-x_{2i})\widehat{f}(x_{2i}-x_{2i+1})$    by
$\widehat{f}( -x_{2i})\widehat{f}(x_{2i} )= |\widehat{f}(  x_{2i})|^2 $. Then, the differences  $\widehat{f}(x_{2i-1}-x_{2i})- \widehat{f}( -x_{2i})$ and
$\widehat{f}(x_{2i}-x_{2i+1})- \widehat{f}(  x_{2i})$ are bounded by constant multiples of $|x_{2i-1}|^\alpha$ and $|x_{2i+1}|^\alpha$, respectively, for any
$0\le \alpha \le 1$, because  $\widehat f(0)=0$ due to the fact that $f$ has zero integral.  We are going to make these substitutions recursively. To do this, we introduce the following notation.

Let $I^n_{m,0}=I^n_m$. For $k=1,\dots,m$, we define
\begin{align*}
I^n_{m,k}
&=c_{m,d}\, n^{m\frac{Hd-1}{2}}\int_{\R^{md}}\int_{D_m} I_k\, \prod^m_{i=k+1} \widehat{f}(x_i-x_{i+1})\\
&\qquad\qquad\times \exp\Big(-\frac{1}{2}\Var\big(\sum\limits^m_{i=1} x_i\cdot (B(s_i)-B(s_{i-1}))\big)\Big)\, ds\, dx,
\end{align*}
where 
\[
I_k =
\begin{cases}
\prod\limits^{\frac{k-1}{2}}_{j=1}|\widehat{f}(x_{2j})|^2 \widehat{f}(-x_{k+1}), & \text{if } k \text{ is odd}; \\
\prod\limits^{\frac{k}{2}}_{j=1}|\widehat{f}(x_{2j})|^2, & \text{if }k\text{ is even}.
\end{cases}
\]
The following proposition controls the difference between $I^n_{m,k-1}$ and $ I^n_{m,k}$. We fix a positive constant $\gamma$ such that
\begin{equation}  \label{gamma}
 \gamma<
\begin{cases}
\frac{1-Hd }{2} & \text{if } 1-Hd\leq H; \\
\frac{2H-1-Hd}{2} & \text{if } H< 1-Hd< 2H.
\end{cases}
\end{equation}

\begin{proposition} \label{prop1} 
For $k=1,2,\dots,m$, there exists a positive constant $c$, which depends on $\gamma$, such that
\[
|I^n_{m,k-1}-I^n_{m,k}|\leq c\, n^{-\gamma} (b-a)^{m\frac{1-Hd}{2}-\gamma}.
\]
\end{proposition}
 
\begin{proof} The proof will be done in several steps.

\noindent \textit{Step 1.} Suppose first that $k=1$. 
Applying the local nondeterminism property  (\ref{lndp}) and making the change of variable $u_1=s_1$, $u_i =s_i -s_{i-1}$, for $2\le i\le m$, we can show that $|I^n_{m,0}-I^n_{m,1}|$ is less than a constant multiple of
\[
n^{m\frac{Hd-1}{2}}\int_{\R^{md}}\int_{O_m} \big|\widehat{f}(x_1-x_2)-\widehat{f}(-x_2)\big| \Big(\prod^m_{i=2}  |\widehat{f}(x_i-x_{i+1})|\Big)\, \exp\Big(-\frac{\kappa_H}{2} \sum\limits^m_{i=1} |x_i|^2u^{2H}_i\Big)\, du\, dx,
\]
where 
\[
O_m=\big\{(u_1, \dots, u_m):\,  0<u_i<n(b-a),\, i=1,\cdots, m \big\}.
\]  

Taking into account that  that $|\widehat{f}(x)|\leq c_{\alpha}|x|^{\alpha}$ for $\alpha\in[0,1]$, we obtain
\begin{align} \label{eq1} \nonumber
|I^n_{m,0}-I^n_{m,1}|
&\leq c_1\, n^{m\frac{Hd-1}{2}}\int_{\R^{md}}\int_{O_m} |x_1|^{\alpha_1} \prod^{m-1}_{i=2} (|x_i|^{\alpha_i}+|x_{i+1}|^{\alpha_i})\, |x_m|^{\alpha_m}\\ \nonumber
&\qquad\qquad\times \exp\Big(-\frac{\kappa_H}{2} \sum\limits^m_{i=1} |x_i|^2u^{2H}_i\Big)\, du\, dx\\
&= c_1\, n^{m\frac{Hd-1}{2}} \sum_{S}  \int_{\R^{md}}\int_{O_m} |x_1|^{\alpha_1} \prod^{m-1}_{i=2} (|x_i|^{p_i\alpha_i}|x_{i+1}|^{\overline{p}_i\alpha_i})\, |x_m|^{\alpha_m}\\ \nonumber
&\qquad\qquad\times \exp\Big(-\frac{\kappa_H}{2} \sum\limits^m_{i=1} |x_i|^2u^{2H}_i\Big)\, du\, dx,  
\end{align}
where $S=\big\{p_i, \overline{p}_i:  p_i\in \{0,1\}, p_i+\overline{p}_i=1,\, i=2,\dots,m-1\big\}$ and the $\alpha_i$s are constants in $[0,1]$.

Rewriting the right hand side of \eref{eq1} gives
\begin{align*}
|I^n_{m,0}-I^n_{m,1}|
&\leq c_1\, n^{m\frac{Hd-1}{2}} \sum_{S} \int_{\R^{md}}\int_{O_m}|x_1|^{\alpha_1}|x_2|^{p_2\alpha_2} \Big(\prod^{m-1}_{i=3} |x_i|^{\overline{p}_{i-1}\alpha_{i-1}+p_i\alpha_i}\Big) |x_m|^{\overline{p}_{m-1}\alpha_{m-1}+\alpha_m}\\
&\qquad\qquad\times \exp\Big(-\frac{\kappa_H}{2} \sum\limits^m_{i=1} |x_i|^2u^{2H}_i\Big)\, du\, dx.
\end{align*}
Integrating with respect to $x$ gives
\begin{align*}
|I^n_{m,0}-I^n_{m,1}|
&\leq c_2\, n^{m\frac{Hd-1}{2}} \sum_{S} \int_{O_m}  u_1 ^{-Hd-H\alpha_1} u_2 ^{-Hd-Hp_2\alpha_2} \prod^{m-1}_{i=3}  u_i ^{-Hd-H(\overline{p}_{i-1}\alpha_{i-1}+p_i\alpha_i)}\\
&\qquad\qquad \times  u_m ^{-Hd-H(\overline{p}_{m-1}\alpha_{m-1}+\alpha_m)}\, du.
\end{align*}
Assume that
\[
1-Hd-H\alpha_1>0, \, 1-Hd-Hp_2\alpha_2>0,\, 1-Hd-H(\overline{p}_{m-1}\alpha_{m-1}+\alpha_m)>0
\]
and
\[
1-Hd-H(\overline{p}_{i-1}\alpha_{i-1}+p_i\alpha_i)>0\quad\text{for}\;\, i=3,\dots, m-1.
\]
Then
\[
|I^n_{m,0}-I^n_{m,1}|
\leq c_3\, n^{m\frac{Hd-1}{2}} (nb-na)^{m(1-Hd)-H\sum\limits^m_{i=1}\alpha_i}.
\]
Fix $\epsilon  >0$. We choose $\alpha_1=  \frac{1-Hd}{H}-\epsilon$  if $1-Hd\leq H$.  Otherwise, we let $\alpha_1=1$.  For $i=2,\dots, m$, we choose $\alpha_i=\frac{1-Hd}{2H}-\epsilon$. With these choices of $\alpha_i$s, we   obtain
\[
  m\frac{Hd-1}{2}+m(1-Hd)-H\sum^m_{i=1}\alpha_i=
\begin{cases}
\frac{Hd-1}{2} + m H\epsilon   & \text{if } 1-Hd\leq H; \\
\frac{1-Hd-2H}{2}  + (m-1) H\epsilon & \text{if } 1-Hd> H.
\end{cases}
\]
Thus we can choose $\epsilon$ such that $ m\frac{Hd-1}{2}+m(1-Hd)-H\sum\limits^m_{i=1}\alpha_i =-\gamma$, and
\[
|I^n_{m,0}-I^n_{m,1}| \leq c_4\, n^{-\gamma} (b-a)^{m\frac{1-Hd}{2}-\gamma},
\]
which is the desired estimation.

\noindent \textit{Step 2:} Suppose now that $k=2$. 
By the definition of $I^n_{m,1}$ and $I^n_{m,2}$, $|I^n_{m,1}-I^n_{m,2}|$ is less than a constant multiple of
\[
n^{m\frac{Hd-1}{2}}\int_{\R^{md}}\int_{O_m} \big|\widehat{f}(-x_2)\big|\big|\widehat{f}(x_2-x_3)-\widehat{f}(x_2)\big| \prod^m_{i=3}  |\widehat{f}(x_i-x_{i+1})|\, \exp\Big(-\frac{\kappa_H}{2} \sum\limits^m_{i=1} |x_i|^2u^{2H}_i\Big)\, du\, dx.
\]
Using similar arguments as in Step 1,
 \begin{align*}
 |I^n_{m,1}-I^n_{m,2}|
&\leq c_5\, n^{m\frac{Hd-1}{2}} \sum_{S_1} \int_{\R^{md}}\int_{O_m} |x_2|^{\alpha_1}|x_3|^{\alpha_2} \Big(\prod^{m-1}_{i=3}  |x_i|^{p_i\alpha_i}|x_{i+1}|^{\overline{p}_i\alpha_i}\Big) |x_m|^{\alpha_m}\\
&\qquad\qquad\times \exp\Big(-\frac{\kappa_H}{2} \sum\limits^m_{i=1} |x_i|^2u^{2H}_i\Big)\, du\, dx\\
&\leq c_6\, \sum_{S_1} n^{m\frac{Hd-1}{2}}\int_{O_m} |u_1|^{-Hd}|u_2|^{-Hd-H\alpha_1}|u_3|^{-Hd-H(\alpha_2+p_3\alpha_3)}\\
&\qquad\qquad \times\Big(\prod^{m-1}_{i=4} |u_i|^{-Hd-H(\overline{p}_{i-1}\alpha_{i-1}+p_i\alpha_i)}\Big) |u_m|^{-Hd-H(\overline{p}_{m-1}\alpha_{m-1}+\alpha_m)}\, du,
\end{align*}
where $S_1=\big\{p_i,\, \overline{p}_i: p_i\in\{0,1\}, p_i+\overline{p}_i=1,\, i=3,\cdots, m-1\big\}$. Then we can conclude as in Step 1.

\noindent \textit{Step 3:}
 Suppose that  $k$ is odd and $3\leq k\leq m$.  Since $k$ is odd, $|I^n_{m,k-1}-I^n_{m,k}|$ is less than a constant multiple of
\begin{align*}
&n^{m\frac{Hd-1}{2}}\int_{\R^{md}}\int_{O_m} \Big(\prod^m_{i=k+1}  |\widehat{f}(x_i-x_{i+1})|\Big)\\
&\qquad\quad\times \big|\widehat{f}(x_k-x_{k+1})-\widehat{f}(-x_{k+1})\big|\Big(\prod^{\frac{k-1}{2}}_{j=1}|\widehat{f}(x_{2j})|^2\Big)\, \exp\Big(-\frac{\kappa_H}{2} \sum\limits^m_{i=1} |x_i|^2u^{2H}_i\Big)\, du\, dx.
\end{align*}
Therefore,   $|I^n_{m,k-1}-I^n_{m,k}|$ is less than a constant multiple of
\[
n^{m\frac{Hd-1}{2}}\int_{\R^{md}}\int_{O_m}  \Big(\prod^m_{i=k+1}  |\widehat{f}(x_i-x_{i+1})|\Big)  |x_{k}|^{\alpha_k}\Big(\prod^{\frac{k-1}{2}}_{j=1}|\widehat{f}(x_{2j})|^2\Big) \exp\Big(-\frac{\kappa_H}{2} \sum\limits^m_{i=1} |x_i|^2u^{2H}_i\Big)\, du\, dx.
\]
Integrating with respect to $x_i$s and $u_i$s with $i\leq k-1$ gives
\begin{align*}
&|I^n_{m,k-1}-I^n_{m,k}|\\
&\leq c_7\, (b-a)^{\frac{k-1}{2}(1-Hd) }\, n^{(m-k+1)\frac{Hd-1}{2}}\int_{\R^{(m-k+1)d}}\int_{O_{m,k}}  \Big(\prod^m_{i=k+1}  |\widehat{f}(x_i-x_{i+1})|\Big) |x_k|^{\alpha_k}\\
&\qquad\qquad\times \exp\Big(-\frac{\kappa_H}{2} \sum\limits^m_{i=k} |x_i|^2u^{2H}_i\Big)\, du\, dx,
\end{align*}
where $du=du_k\cdots du_{m}$, $dx=dx_k\cdots dx_{m}$ and 
\[
O_{m,k}=\big\{ (u_k, \dots, u_m):   0\le u_i \le n(b-a),    i=k,\dots,m\big\}.
\]
Applying Step 1 and then doing some algebra, we can obtain
\[
|I^n_{m,k-1}-I^n_{m,k}|\leq c_8\, n^{-\gamma} (b-a)^{m\frac{1-Hd}{2}-\gamma}.
\]

\noindent \textit{Step 4:} The case when $k$ is even and $4\leq k\leq m$ is handled in a similar way.
\end{proof}

\section{Proof of Theorem \ref{th1}}
  The proof  of Theorem \ref{th1} will be done in two steps. We first
show tightness, and then establish the convergence of moments. Tightness will be deduced from the following inequality.

\begin{proposition} \label{tight}
For any $0\leq a<b \leq t$ and any integer $m\geq 1$,
\begin{equation*}
\E\big[(F_{n}(b)-F_{n}(a))^{2m}\big] \leq C\,
(b-a)^{m(1-Hd)-\gamma},
\end{equation*}
where $C$ is a constant depending only on $H$, $m$, $d$ and $f$.
\end{proposition}
\begin{proof} Note that $\E\big[(F_{n}(b)-F_{n}(a))^{2m}\big] =I^n_{2m,0}$. Applying Proposition \ref{prop1} repeatedly gives 
\begin{align} \label{tight1}
I^n_{2m,0}\leq c_1\, n^{-\gamma} (b-a)^{m(1-Hd)-\gamma}+c_1\, I^n_{2m,2m}.
\end{align}
So it suffices to estimate $I^n_{2m,2m}$. By the definition of $I^n_{2m,2m}$, using the same notation as in the proof of Proposition \ref{prop1}, we obtain
\begin{align}  \nonumber
I^n_{2m,2m} 
&=c_2\, n^{m(1-Hd)}\int_{\R^{2md}}\int_{D_{2m}}\Big(\prod^m_{j=1}|\widehat{f}(x_{2j})|^2\Big)\\  \nonumber
&\qquad\qquad\times \exp\bigg(-\frac{1}{2}\Var\Big(\sum\limits^{2m}_{i=1} x_i\cdot\big(B(s_i)-B(s_{i-1})\big)\Big)\bigg)\, ds\, dx\\  \nonumber
&\leq c_3\, n^{m(1-Hd)}\int_{\R^{2md}}\int_{O_{2m}}\Big(\prod^m_{j=1}|\widehat{f}(x_{2j})|^2\Big)\, \exp\Big(-\frac{\kappa_H}{2}\sum\limits^{2m}_{i=1} |x_i|^2u^{2H}_i\Big)\, du\, dx\\ 
&\leq c_4\, (b-a)^{m(1-Hd)}\Big(\int_{\R^d}|\widehat{f}(x)|^2|x|^{-\frac{1}{H}}\, dx\Big)^m. \label{tight2}
\end{align}
Combining \eref{tight1} and \eref{tight2} gives the desired result.
\end{proof}

\bigskip

Next we shall prove the convergence of all finite dimensional distributions. That is, we shall prove that the moments of $F_n(t)$ converge to the
corresponding ones of $W(L_t(0))$.

Fix a finite number of disjoint intervals $(a_i, b_i]$ with
$i=1,\dots ,N$ and
 $b_i\le a_{i+1}$.  Let $\mathbf{m}=(m_1, \dots, m_N)$ be a fixed multi-index with $m_i\in\N$ for $i=1,\dots ,N$.
 Set $\sum\limits_{i=1}^N m_i=|\mathbf{m}|$ and  $\prod\limits_{i=1}^N m_i!=\mathbf{m}!$.
We need to consider the following sequence of random variables
 \[
        G_n=\prod_{i=1}^N \left( F_n(b_i)- F_n(a_i) \right)^{m_i}
 \]
and compute  $\lim\limits_{n\rightarrow  \infty}  \E(G_n) $. Note that the expectation of $G_n$ can be formulated as 
 \[
      \E(G_n) = \mathbf{m}!\, n^{\frac{|\mathbf{m}|(Hd-1)} 2}  \E \Big(
           \int_{D_\mathbf{m}}  \prod_{i=1}^N\prod_{j=1}^{m_i}f(B(s^i_{j}))\, ds \Big),
 \]
where 
\begin{equation}
D_\mathbf{m}=\big\{s \in \R^{|\mathbf{m}|}: na_{i}<s^i_{1}<\cdots <s^i_{m_i}<nb_{i}, 1\le i\le N\big\}. \label{e.3.1.4}
\end{equation}
Here and in the sequel we denote the coordinates of a point $s\in  \R^{|\mathbf{m}|}$ as
$s= (s^i_j)$, where   $ 1\le i \le   N$ and $1\le j \le m_i $.

For simplicity of notation, we define 
\[
J_0=\big\{(i,j): 1\leq i\leq N, 1\leq j\leq m_i \big\}.
\]
For any $(i_1, j_1)$ and $(i_2,j_2)\in J_0$, we define the following dictionary ordering 
\[
(i_1, j_1)\leq (i_2,j_2)
\]
if $i_1<i_2$ or $i_1=i_2$ and $j_1\leq j_2$. For any $(i,j)$ in $J_0$, under the above ordering, $(i,j)$ is the $(\sum\limits^{i-1}_{k=1}m_k+j)$-th element in $J_0$ and we define $\#(i,j)=\sum\limits^{i-1}_{k=1}m_k+j$.

\begin{proposition} \label{odd} Suppose that at least one of the exponents $m_i$ is odd. Then
\begin{equation*}
\lim\limits_{n\to\infty}\E(G_n)=0.
\end{equation*}
\end{proposition}
\begin{proof}  Using Fourier transform, we see that $\E (G_n)$ is equal to
\begin{align*}
 \frac{\mathbf{m}!}{(2\pi)^{|\mathbf{m}|d}}\,  n^{\frac{|\mathbf{m}|(Hd-1)}{2}}  \int_{\R^{|\mathbf{m}|d}}  \int_{D_\mathbf{m}} 
           \Big( \prod^{N}_{i=1}\prod^{m_i}_{j=1} \widehat{f}(y^i_j)\Big) \exp\bigg(-\frac{1}{2}\Var\Big( \sum^{N}_{i=1}\sum^{m_i}_{j=1} y^i_j\cdot B(s^i_j) \Big)\bigg) \, ds\, dy.
\end{align*}
Making the
change of variables $x^i_j=\sum\limits_{(\ell,k)\geq (i,j)} y^{\ell}_k$ for $1\leq i \leq N$ and $1\leq j\leq m_i$,
\begin{align*}
\E (G_n)
&=\frac{\mathbf{m}!}{(2\pi)^{|\mathbf{m}|d}}\,  n^{\frac{|\mathbf{m}|(Hd-1)}{2}}  \int_{\R^{|\mathbf{m}|d}}  \int_{D_\mathbf{m}} 
            \prod^{N}_{i=1}\prod^{m_i}_{j=1} \widehat{f}(x^i_j-x^i_{j+1})\\
&\qquad\qquad\times \exp\bigg(-\frac{1}{2}\Var\Big( \sum^{N}_{i=1}\sum^{m_i}_{j=1} x^i_j\cdot \big(B(s^i_j)-B(s^i_{j-1})\big) \Big)\bigg) \, ds\, dx.
\end{align*}
Applying Proposition \ref{prop1},  we   obtain
\begin{align*}
\lim_{n\to\infty}\E (G_n)
&=\frac{\mathbf{m}!}{(2\pi)^{|\mathbf{m}|d}}\,  \lim_{n\to\infty} n^{\frac{|\mathbf{m}|(Hd-1)}{2}}  \int_{\R^{|\mathbf{m}|d}}  \int_{D_\mathbf{m}} 
           \bigg( \prod_{(i,j)\in J_e} |\widehat{f}(x^i_j)|^2\bigg)\, I_{|\mathbf{m}|}\\
&\qquad\qquad\times \exp\bigg(-\frac{1}{2}\Var\Big( \sum^{N}_{i=1}\sum^{m_i}_{j=1} x^i_j\cdot \big(B(s^i_j)-B(s^i_{j-1})\big) \Big)\bigg) \, ds\, dx,
\end{align*}
where $J_e=\big\{(i,j)\in J_0: \#(i,j)\; \text{is even}\big\}$ and 
\[
I_{|\mathbf{m}|}=
\begin{cases}
\widehat{f}(x^N_{m_N}), & \text{if}\; |\mathbf{m}| \; \text{is odd};\\
1, & \text{if}\; |\mathbf{m}| \; \text{is even}.
\end{cases}
\]

It is easy to see that $\lim\limits_{n\to\infty}\E (G_n)=0$ when $|\mathbf{m}|$ is odd. We shall show $\lim\limits_{n\to\infty}\E (G_n)=0$ when $|\mathbf{m}|$ is even. In this case,
\begin{align*}
\lim_{n\to\infty}\E (G_n)
&=\frac{\mathbf{m}!}{(2\pi)^{|\mathbf{m}|d}}\,  \lim_{n\to\infty} n^{\frac{|\mathbf{m}|(Hd-1)}{2}}  \int_{\R^{|\mathbf{m}|d}}  \int_{D_\mathbf{m}} 
           \Big( \prod_{(i,j)\in J_e} |\widehat{f}(x^i_j)|^2\Big)\\
&\qquad\qquad\times \exp\bigg(-\frac{1}{2}\Var\Big( \sum^{N}_{i=1}\sum^{m_i}_{j=1} x^i_j\cdot \big(B(s^i_j)-B(s^i_{j-1})\big) \Big)\bigg) \, ds\, dx.
\end{align*}
Note that the right hand side of the above equality is positive. Using the local nondeterminism property (\ref{lndp}),
\begin{align*}
\big|\lim_{n\to\infty}\E (G_n)\big|
&\leq c_1  \limsup_{n\to\infty} n^{\frac{|\mathbf{m}|(Hd-1)}{2}}  \int_{\R^{|\mathbf{m}|d}}  \int_{D_\mathbf{m}} 
           \Big( \prod_{(i,j)\in J_e} |\widehat{f}(x^i_j)|^2\Big)\\
&\qquad\qquad\times \exp\bigg(-\frac{\kappa_H}{2}\sum^{N}_{i=1}\sum^{m_i}_{j=1} |x^i_j|^2(s^i_j-s^i_{j-1})^{2H}\bigg) \, ds\, dx\\
&:= c_1 \limsup_{n\to\infty} I_n.
\end{align*}

Assume that $m_{\ell}$ is the first odd exponent. Integrating with respect to proper $x^i_j$s and $s^i_j$s gives
\begin{align*}
I_n &\leq c_2\, n^{Hd-1} \sup_{s^{\ell}_{m_{\ell}-1}\in (na_\ell,nb_{\ell}]} \int_{\R^{d}}  \int^{nb_{\ell}}_{s^{\ell}_{m_{\ell}-1}}\int^{nb_{\ell+1}}_{n_{a_{\ell+1}}}
           |\widehat{f}(x^{\ell+1}_1)|^2(s^{\ell}_{m_{\ell}}-s^{\ell}_{m_{\ell}-1})^{-Hd}\\
&\qquad\qquad\times   \exp\Big(-\frac{\kappa_H}{2} |x^{\ell+1}_1|^2(s^{\ell+1}_1-s^{\ell}_{m_{\ell}})^{2H}\Big) \, ds^{\ell+1}_1\, ds^{\ell}_{m_{\ell}}\, dx^{\ell+1}_1.
\end{align*}
Note that $|\widehat{f}(x)|\leq c_{\alpha}|x|^{\alpha}$ for $\alpha\in[0,1]$. Choosing $\alpha\in(\frac{1-Hd}{2H},\frac{2-Hd}{2H})$ gives
\begin{align*}
I_n &\leq c_3\, n^{Hd-1} \sup_{s^{\ell}_{m_{\ell}-1}\in (na_\ell,nb_{\ell}]} \int_{\R^{d}}  \int^{nb_{\ell}}_{s^{\ell}_{m_{\ell}-1}}\int^{nb_{\ell+1}}_{n_{a_{\ell+1}}}
           |x^{\ell+1}_1|^{2\alpha}(s^{\ell}_{m_{\ell}}-s^{\ell}_{m_{\ell}-1})^{-Hd}\\
&\qquad\qquad\times   \exp\Big(-\frac{\kappa_H}{2} |x^{\ell+1}_1|^2(s^{\ell+1}_1-s^{\ell}_{m_{\ell}})^{2H}\Big) \, ds^{\ell+1}_1\, ds^{\ell}_{m_{\ell}}\, dx^{\ell+1}_1\\ 
&\leq c_4\, n^{Hd-1} \sup_{s^{\ell}_{m_{\ell}-1}\in (na_\ell,nb_{\ell}]} \int^{nb_{\ell}}_{s^{\ell}_{m_{\ell}-1}}\int^{nb_{\ell+1}}_{n_{a_{\ell+1}}}
           (s^{\ell+1}_1-s^{\ell}_{m_{\ell}})^{-Hd-2H\alpha}(s^{\ell}_{m_{\ell}}-s^{\ell}_{m_{\ell}-1})^{-Hd}\, ds^{\ell+1}_1\, ds^{\ell}_{m_{\ell}}\\    
&\leq c_5\, n^{Hd-1} \sup_{s^{\ell}_{m_{\ell}-1}\in (na_\ell,nb_{\ell}]} \int^{nb_{\ell}}_{s^{\ell}_{m_{\ell}-1}}
           (na_{\ell+1}-s^{\ell}_{m_{\ell}})^{1-Hd-2H\alpha}(s^{\ell}_{m_{\ell}}-s^{\ell}_{m_{\ell}-1})^{-Hd}\, ds^{\ell}_{m_{\ell}}\\     
&\leq c_6\, n^{1-Hd-2H\alpha},               
\end{align*}
where we used $b_{\ell}\leq a_{\ell+1}$ in the last inequality. Therefore, 
\[
\big|\lim\limits_{n\to\infty}\E (G_n)\big|\leq c_6\lim_{n\to\infty} n^{1-Hd-2H\alpha}=0.
\]
\end{proof}

\bigskip
 Consider now  the convergence of moments when all exponents $m_i$ are even.  
 \begin{proposition} \label{even} Suppose that all exponents $m_i$ are even. Then
 \begin{equation}  \label{c1}
\lim_{n\to\infty}\E(G_n)=C_{H,d}^{\frac{|\mathbf{m}|}{2}}\, \Vert
f\Vert^{ |\mathbf{m}| }_{\frac{1}{H}-d }\,\, \mathbb{E} \Big(
\prod_{i=1}^N   \big( W(L_{b_i}(0))- W(L_{a_i}(0)) \big)^{m_i}
\Big),
\end{equation}
where the expectation in the right-hand side of the above equation is given by formula  (\ref{2.1}).
\end{proposition}
\begin{proof}
For any $K>0$ and $\ell=1,\dots, |\mathbf{m}|/2$,  we introduce the set 
\[
D_{|\mathbf{m}|,K}^\ell= \big\{(s_1, \dots, s_{|\mathbf{m}|}) \in D_{\mathbf{m}}:\, s_1<s_2<\cdots<s_{|\mathbf{m}|},\, s_{2\ell} -s_{2\ell-1} > K \big\},
\]
where $ D_{\mathbf{m}}$ is defined in \eref{e.3.1.4}.

Taking into account the results proved in \cite{HNX}, the proof of the convergence  (\ref{c1}) reduces to show that   for all $\ell=1,\dots,|\mathbf{m}|/2$
 \begin{equation} \label{c2} 
 \lim_{K\to\infty}\limsup_{n\to\infty}\, n^{\frac{ |\mathbf{m}|(Hd-1)}2}\, 
 \E \Big(
           \int_{D_{|\mathbf{m}|,K}^\ell} \prod_{i=1}^{|\mathbf{m}|}f(B(s_i))\, ds \Big)=0.
  \end{equation}         
 In order to prove (\ref{c2}), set  $|\mathbf{m}|=2m$. Then, it suffices to show that
\[
\lim_{K\to\infty}\limsup_{n\to\infty}\, n^{m(Hd-1)}\int_{\R^{2md}}\int_{D^{\ell}_{2m,K}} \Big(\prod^{2m}_{i=1} \widehat{f}(x_i)\Big) \exp\Big(-\frac{1}{2}\Var\big(\sum\limits^{2m}_{i=1} x_i\cdot B(s_i)\big)\Big)\, ds\, dx=0.
\]
Using similar arguments as in the proof Proposition  \ref{prop1} we can write
\begin{align*}
&\limsup_{n\to\infty}\, n^{m(Hd-1)}\int_{\R^{2md}}\int_{D^{\ell}_{2m,K}} \Big(\prod^{2m}_{i=1} \widehat{f}(x_i)\Big)\, \exp\Big(-\frac{1}{2}\Var(\sum\limits^{2m}_{i=1} x_i\cdot B(s_i))\Big)\, ds\, dx\\
=&\limsup_{n\to\infty}\, n^{m(Hd-1)}\int_{\R^{2md}}\int_{D^{\ell}_{2m,K}} \Big(\prod^{m}_{j=1} |\widehat{f}(z_{2j})|^2\Big)\\
&\qquad\qquad\times \exp\bigg(-\frac{1}{2}\Var\Big(\sum\limits^{2m}_{i=1} z_i\cdot \big(B(s_i)-B(s_{i-1})\big)\Big)\bigg)\, ds\, dz.
\end{align*}
The right hand side of the above equality is positive and less  than or equal to
\[
\limsup_{n\to\infty}\, n^{m(Hd-1)}\int_{\R^{2md}}\int_{D^{\ell}_{2m,K}} \Big(\prod^{m}_{j=1} |\widehat{f}(z_{2j})|^2\Big)\, \exp\Big(-\frac{\kappa_H}{2}\sum\limits^{2m}_{i=1} |z_i|^2(s_i-s_{i-1})^{2H}\Big)\, ds\, dz.
\]
Integrating with respect all $z_i$s and $s_i$s ($i\neq 2\ell$), the above limit is less than or equal to 
\begin{align*}
&c_1\int_{\R^d}\int^{\infty}_K |\widehat{f}(z_{2\ell})|^2\, e^{-\frac{\kappa_H}{2}|z_{2\ell}|^2u^{2H}}\, du\, dz_{2\ell}\\
\leq &\, c_2 \int_{\R^d}\int^{\infty}_K |z_{2\ell}|^2\, e^{-\frac{\kappa_H}{2}|z_{2\ell}|^2u^{2H}}\, du\, dz_{2\ell}\\
= &\, c_3\, K^{1-Hd-2H}.
\end{align*}
This completes the proof since $1-Hd-2H<0$.
 \end{proof}

{\medskip \noindent \textbf{Proof of Theorem \ref{th1}.}}    This follows
from Lemma \ref{lema2}, Propositions \ref{tight}, \ref{odd} and \ref{even} by the method of moments.


\begin{thebibliography}{99}

\bibitem{Be} Berman, S.M.  Local nondeterminism and local times of Gaussian
processes. \textit{Indiana Univ. Math.} \textbf{23} (1973)  64--94.


 

\bibitem{GH} Geman, D. and Horowitz, J.  Occupation densities. \textit{Annals of
Probability} \textbf{8} (1980)  1-67.

 
\bibitem{HNX} Hu, Y., Nualart, D., and Xu, F.  Central limit theorem for an additive functional of the fractional Brownian motion. \textit{Annals of Probability}, accepted.

 \bibitem{PSV}
Papanicolaou, G.C., Stroock, D., Varadhan, S. R. S. Martingale approach to some limit theorems.
\textit{Duke Univ. Math. Ser. III, Statistical Mechanics and Dynamical Systems}, 1977.


 




\end{thebibliography}
\end{document}